\documentclass[10pt,a4paper]{article}

\usepackage{hyperref}
\usepackage[spanish,british]{babel}
\usepackage[utf8]{inputenc}
\usepackage{amsmath}
\usepackage{amsfonts}
\usepackage{amssymb}
\usepackage{amsthm}
\usepackage{graphicx}
\usepackage{fancyhdr}
\usepackage{makeidx}
\usepackage{tikz}
\usepackage{subfigure}
\usepackage{hyperref}
\usepackage{dsfont}
\usepackage{parskip}
\usepackage{aliascnt}
\usepackage{enumerate}
\usepackage{mathrsfs}

\newcommand{\F}{\mathcal{F}}
\newcommand{\FF}{\{\F_t\colon t\geq 0\}}
\newcommand{\T}{\bf T}
\newcommand{\R}{\mathbb{R}}

\newcommand{\Ga}{G_{\alpha}}
\newcommand{\Ma}{M_{\alpha}}
\newcommand{\Ra}{R_{\alpha}}
\newcommand{\Va}{V_{\alpha}}
\newcommand{\Vat}{\tilde{V}_{\alpha}}
\newcommand{\Wa}{W_{\alpha}}
\newcommand{\phia}{\varphi_\alpha}
\newcommand{\psia}{\psi_\alpha} 
\newcommand{\wa}{w_\alpha}
\newcommand{\al}{(\alpha-L)}
\newcommand{\hit}[1]{\tau_{#1}}

\newcommand{\g}{{g}}
\newcommand{\gi}{\tilde{g}}
\newcommand{\D}{\mathcal{D}}
\newcommand{\CC}{\mathcal{C}}
\newcommand{\I}{\mathcal{I}}
\newcommand{\ea}[1]{e^{-\alpha{#1}}}
\newcommand{\sfp}{SF}
\newcommand{\ssfp}{SSF}  
\newcommand{\Ex}[2]{\E_{#1}\left(#2\right)}

\def\P{\operatorname{\bf P}}
\def\E{\operatorname{\bf E}}

\newtheorem{theorem}{Theorem}[section]

\newtheorem{corollary}[theorem]{Corolary}
\newtheorem{lemma}[theorem]{Lemma}
\newtheorem{proposition}[theorem]{Proposition}

\theoremstyle{definition}
\newtheorem{remark}[theorem]{Remark}

\begin{document}
\title{Explicit solutions in one-sided optimal stopping problems for one-dimensional diffusions }
\author{Fabi\'an Crocce$^*$ and Ernesto  Mordecki\footnote{Igu\'a 4225, Centro de Matem\'atica, Facultad de Ciencias, Universidad de la Rep\'ublica, Montevideo, Uruguay}} 

\maketitle
\begin{abstract}

Consider the optimal stopping problem of a one-dimensional diffusion 
with positive discount.
Based on Dynkin's characterization of the value   
as the minimal excessive majorant of the reward
and considering its Riesz representation, 
we give an explicit equation to find the optimal stopping threshold for problems with one-sided stopping regions, 
and an explicit formula for the value function of the problem. 
This representation also gives light on the validity of the smooth fit principle.
The results are illustrated by solving some classical problems,  
and also through the solution of: optimal stopping of the skew Brownian motion, and optimal stopping of the sticky Brownian motion, including cases in which the smooth fit principle fails.
\end{abstract}

%
%
\section{Introduction and problem formulation}
Consider a non-terminating and regular one-dimensional (or linear) diffusion $X=\{X_t\colon t\geq 0\}$, 
in the sense of It\^o and McKean \cite{itoMcKean} (see also Borodin and Salminen \cite{borodin}). The state space of $X$ is denoted by $\I$, 
an interval of the real line $\R$  with left endpoint $\ell=\inf\I$ and right endpoint $r=\sup\I$,
where $-\infty\leq\ell<r\leq\infty$. 
We exclude the possibility of absorbing and killing boundaries; if some of the boundaries belong to $\I$ we assume it to be both entrance and exit (i.e. reflecting boundary).
Denote by $\P_x$ the probability measure associated with $X$ when starting from $x$, 
and by $\E_x$ the corresponding mathematical expectation. 
Denote by $\T$ the set of all stopping times with respect to $\FF$, the usual augmentation of the natural 
filtration generated by $X$ (see I.14 in \cite{borodin}).

Given a non-negative lower semicontinuous reward function $\g\colon\I\to \R$ and a discount factor $\alpha >0$, 
consider the optimal stopping problem consisting in finding a function $\Va$ and a stopping time $\tau^*\in\T$, such that
\begin{equation}\label{eq:osp}
\Va(x)=\Ex{x}{\ea{\tau^*} \g(X_{\tau^*})} = \sup_{\tau \in \T}\Ex{x}{ \ea{\tau^*} \g(X_{\tau^*})}.
\end{equation}
The \emph{value function} $\Va$ and the \emph{optimal stopping time} $\tau^*$  are the solution of the problem.
 
The first problems in optimal stopping appeared in the framework of statistics, 
more precisely, in the context of sequential analysis (see  the book by Wald \cite{wald}). 
For continuous time processes a relevant reference is the book of Shiryaev \cite{shiryaev} 
that also has applications to statistics.
A new impulse to these problems is related to mathematical finance, 
where arbitrage considerations give that 
in order to price an American option one has to solve an optimal stopping problem. 
The first results in this direction were provided by Mc Kean \cite{mckean} in 1965 and Merton \cite{merton} in 1973, 
who respectively solved the perpetual put and call option pricing problem, by solving the corresponding optimal stopping problems in the context of the Black and Scholes model \cite{bs}.
For the state of the art in the subject see the book by Peskir and Shiryaev \cite{ps} and the references therein. A new approach for solving one-dimensional optimal stopping problems for very general reward functions is provided in the recent paper \cite{peskir2012duality}.

When considering optimal stopping problems we typically find two classes of results. The first one consists in the explicit solution to a concrete optimal stopping problem \eqref{eq:osp}. Usually in this case one has to --somehow-- guess the solution and prove that this guess in fact solves the optimization problem; we call this approach ``verification''. For example we can consider the papers \cite{mckean}, \cite{merton}, \cite{taylor}, \cite{shepp93}. The second class consists of general results, for wide classes of processes and reward functions. We call this the ``theoretical'' approach. It typically include results about properties of the solution. In this class we mention for example \cite{dynkin:1963}, \cite{grigelionis1966stefan}, \cite{elkaroui}.
But these two classes not always meet, as frequently in concrete problems the assumptions of the theoretical studies are not fulfilled, and, what is more important, many of these theoretical studies do not provide concrete ways to find solutions. In what concerns the first approach, a usual procedure is to apply the \emph{principle of smooth fit}, that 
generally leads to the solution of two equations: 
the \emph{continuous fit}  equation and the \emph{smooth fit} equation. 
Once these equations are solved, 
a verification procedure is needed in order to prove that the candidate is the effective solution of the problem (see chapter IV in \cite{ps}). This approach, when an explicit solution can be found, is very effective.
In what concerns the second approach, maybe the most important result is Dynkin's characterization of the solution of the value function $\Va$ as the least $\alpha$-excessive (or $\alpha$-superharmonic) majorant of the payoff function $\g$ \cite{dynkin:1963}. Other ways of classifying approaches in order to study optimal stopping problems include the martingale-Markovian dichotomy
as exposed in \cite{ps}.

In the present work we provide an explicit solution of a 
\emph{right-sided} optimal stopping problem for a one-dimensional diffusion process, 
i.e., when the optimal stopping time has the form 
\begin{equation}\label{eq:call}
\tau^*=\inf\{t\geq 0\colon X_t\geq x^*\}
\end{equation}
for some optimal threshold $x^*\in\I$, under mild regularity conditions. 
Right-sided problems are also known as \emph{call-type} optimal stopping problems. Analogous results are valid for left-sided problems.

An important byproduct of our results has to do with the smooth fit principle.
Our results are independent of this principle, 
but they give sufficient conditions in order to 
guarantee it.

In section \ref{section:definitions} some necessary definitions  and preliminary results are given. 
Our main results are presented in section \ref{section:main}.
In section \ref{section:sf} we discuss the consequences of our results related to the smooth fit principle.
Finally, in section \ref{section:examples} we present some examples, including the optimal stopping of the skew Brownian motion and of the sticky Brownian motion (suggested by Paavo Salminen), where particular attention to the smooth fit principle is given.

\section{Definitions and preliminary results}\label{section:definitions}

Denote by $L$ the \emph{infinitesimal generator} of the diffusion $X$, and by $\D_L$ its domain.
For any stopping time $\tau$ and for any $f\in \D_L$ 
 the following  discounted version of the Dynkin's formula holds:
\begin{equation}
\label{eq:dynkinFormula}
 f(x)=\Ex{x}{\int_0^{\tau} \ea{t} \al f(X_t) dt}+ \Ex{x}{\ea{\tau}f(X_\tau)}.
\end{equation}

The \emph{resolvent} of the process $X$ is the operator $\Ra$ defined by
\begin{equation*}
\Ra u(x)=\int_0^\infty e^{-\alpha t}\Ex{x}{u(X_t)}dt,
\end{equation*}
applied to a function $u\in \CC_b(\I)=\{u\colon\I\to \R, u \mbox{ is continuous and bounded}\}$. The range of the operator $\Ra$ is independent of  $\alpha>0$ and coincides with the domain of the infinitesimal generator $\D_L$. Moreover, for any $f\in \D_L$, $\Ra \al f = f$, and for any $u\in \CC_b(\I)$, $\al \Ra u = u$. 
In other terms, $\Ra$ and $\al$ are inverse operators.
Denoting by $s$ and $m$ the scale function and the speed measure of the diffusion $X$ respectively, we have that,
for any $f \in \D_L$,
the lateral derivatives with respect to the scale function exist for every $x\in (\ell,r)$.
Furthermore, they satisfy 
\begin{equation}\label{eq:atom} 
\frac{\partial^+ f}{\partial s}(x)- \frac{\partial^- f}{\partial s}(x)=  m(\{x\}) Lf(x),
\end{equation}
and the following identity holds for $z>y$:
\begin{equation}
\label{difRightDer}
 \frac{\partial^+ f}{\partial s}(z)-\frac{\partial^+ f}{\partial s}(y)=\int_{(y,z]} Lf(x) m(dx).
\end{equation}
This last formula allows us to compute the infinitesimal generator of $f$ at $x\in(\ell,r)$ as Feller's differential operator \cite{feller1957generalized}
\begin{equation}
\label{eq:difop}
Lf(x)=\frac{\partial}{\partial m} \frac{\partial^+}{\partial s}f(x).
\end{equation}
The infinitesimal generator at $\ell$ and $r$ (if they belong to $\I$) can be computed as
$Lf(\ell)=\lim_{x\to \ell^+}Lf(x)$ and $Lf(r)=\lim_{x\to r^-}Lf(x)$ respectively.

Given a function $u\colon \I \to \R$, and $x\in (\ell,r)$ we give to $Lu(x)$ the meaning given in \eqref{eq:difop} if it makes sense. 
We also define, if $\ell\in \I$,  $Lu(\ell)=\lim_{x\to \ell^+}Lu(x)$ and if $r\in \I$,  $Lu(r)=\lim_{x\to r^-}Lu(x)$, if the limit exists.

There exist two continuous functions $\phia \colon \I \mapsto \R^+$ decreasing,
and $\psia\colon \I\mapsto \R^+$ increasing, solutions of $\alpha u = Lu$, 
such that any other continuous function $u$ is a solution of the differential equation if and only if 
$u=a\phia+b\psia$, with $a,b$ in $\R$. 
Denoting by $\hit{z}=\inf \{t\colon X_t = z\}$ the hitting time of level $z\in\I$, we have 
	\begin{equation}
	\label{eq:hitting}
	\Ex{x}{e^{-\alpha \hit{z}}}=
	\begin{cases}
	\frac{\psia(x)}{\psia(z)},\quad x\leq z,\\[.5em]
	\frac{\phia(x)}{\phia(z)},\quad x\geq z.
	\end{cases}
	\end{equation}
The functions  $\phia$ and $\psia$, though not necessarily in $\D_L$, also satisfy \eqref{eq:atom}
for all $x\in (\ell,r),$ which allow us to conclude that in case $m(\{x\})=0$, 
the derivative at $x$ of both functions with respect to the scale  exists. 
From \eqref{difRightDer} applied to $\psia$, and taking into account  $\alpha \psia=L\psia$ we obtain for $z>y$
\begin{equation*}
\frac{\partial^+ \psia}{\partial s}(z)-\frac{\partial^+ \psia}{\partial s}(y)=\int_{(y,z]} \alpha \psia(x) m(dx);
\end{equation*}
the right hand side is strictly positive since $\alpha\psia$ is positive and $m$ charge every open set. We conclude that $\frac{\partial^+ \psia}{\partial s}$ is strictly increasing. In an analogous way it can be proven that $\frac{\partial^+ \phia}{\partial s}$ is increasing as well. The previous consideration, together with the fact that $\frac{\partial^+ \psia}{\partial s}(x)\geq 0$ and $\frac{\partial^+ \phia}{\partial s}(x)\leq 0$ allow us to conclude for $x\in (\ell,r)$:
\begin{equation*}
-\infty<\frac{\partial^- \phia}{\partial s}(x) \leq \frac{\partial^+ \phia}{\partial s}(x)<0<\frac{\partial^- \psia}{\partial s}(x) \leq \frac{\partial^+ \psia}{\partial s}(x)<\infty.
\end{equation*}

The \emph{Green function} of the process $X$ with discount factor $\alpha$ is defined by
\begin{equation*}
\Ga(x,y)=\int_0^\infty e^{-\alpha t}p(t;x,y)dt,
\end{equation*}
where $p(t;x,y)$ is the transition density of the diffusion with respect to the speed measure $m(dx)$
(this density always exists, see \cite{itoMcKean} or \cite{borodin}). 
The Green function may be expressed in terms of $\phia$ and $\psia$ as follows:
\begin{equation}
\label{eq:Garepr}
G_\alpha(x,y)=
\begin{cases}
w_\alpha ^{-1} \psia(x) \phia (y),\quad  & x\leq y, \\
w_\alpha ^{-1} \psia(y) \phia (x),  & x\geq y,
\end{cases}
\end{equation}
where $w_\alpha$ is the \emph{ Wronskian}, given by
\begin{equation*}
 w_\alpha=\frac{\partial \psia^+}{\partial s}(x)\phia(x)-\psia(x)\frac{\partial \phia^+}{\partial s}(x).
\end{equation*}
Observe that the Wronskian is positive and independent of $x$ \cite{itoMcKean},\cite{borodin}. 
Given $u\colon\I\to \R$, under the condition $\int_{\I} \Ga(x,y) |u(y)| m(dy)<\infty$, an application of Fubini's Theorem gives
\begin{equation}
\label{eq:Ra=Ga}
\Ra u(x)=  \int_{\I} \Ga(x,y) u(y) m(dy).
\end{equation}

A  non-negative Borel function $u\colon\I\to \R$ is called \emph{$\alpha$-excessive} for the process $X$ if
$e^{-\alpha t}\Ex{x}{u(X_t)}\leq u(x)$ for all $x\in \I$ and $t\geq 0$, and
 $\lim_{t \to 0} \Ex{x}{u(X_t)}= u(x)$ for all $x\in \I$.
A 0-excessive function is said to be \emph{excessive}. 

Consider the process killed at an independent exponential time of parameter $\alpha$, i.e. 
$$Y_t=\begin{cases}
X_t,& t<e_{\alpha}\\
\Delta,& \text{else}
\end{cases} 
$$
with $e_{\alpha}$ a random variable with exponential distribution of parameter $\alpha$ independent of $X$, and $\Delta$ the cemetery state, at which any function is defined to be zero.
It is easy to see that the Green function $G^Y$ of the process $Y=\{Y_t\colon t\geq 0\}$ coincides with $\Ga$; a Borel function $u\colon \I \mapsto \R$ is excessive for $Y$ if and only if it is $\alpha$-excessive for $X$.
In fact, the non-discounted optimal stopping problem for the process $Y$ has the very same solution 
(value function and optimal stopping time) as the $\alpha$-discounted optimal stopping problem for $X$.

For general reference on diffusions and Markov processes see 
\cite{borodin,itoMcKean,revuz,dynkin:books,ks}.

\section{Main results}\label{section:main}

Our departing point, inscribed in the Markovian approach, 
is Dynkin's characterization of the optimal stopping problem solution. Dynkin's characterization \cite{dynkin:1963} states that, if the reward function is lower semi-continuous,  $V$ is the value function of the non-discounted optimal stopping problem with reward $g$ if and only if $V$ is the least excessive function such that $V(x)\geq g(x)$ for all $x\in \I$. Applying this result for the killed process $Y$, and taking into account the relation between $X$ and $Y$, we obtain that $\Va$, the value function of the problem with discount $\alpha$, is characterized as the least $\alpha$-excessive majorant of $g$.

The second step uses Riesz's decomposition of an $\alpha$-excessive function. We recall this decomposition in our context
(see \cite{kw,kw1,dynkin:1969}). 
A function $u\colon \I \to \R$ is $\alpha$-excessive if and only if there exist a non-negative Radon measure $\mu$ and an $\alpha$-harmonic function such that
\begin{equation}
\label{eq:alphaexcessive}
 u(x)=\int_{(\ell,r)}\Ga(x,y)\mu(dy) + \text{($\alpha$-harmonic function)}.
\end{equation}
Furthermore, the previous representation is unique. The measure $\mu$ is called the representing measure of $u$.

The third step is based on the fact that the resolvent and the infinitesimal generator of  a Markov process are inverse operators. 
Suppose that we can write 
\begin{equation}\label{eq:inversionV}
\Va(x)=\int_{\I} \Ga(x,y)\al \Va(y)m(dy),
\end{equation}
where $L$ is the infinitesimal generator and $m(dy)$ is the speed measure of the diffusion. Assuming that the stopping region has the form $\I\cap\{x\geq x^*\}$, and taking into account that $\Va$ is $\alpha$-harmonic in the continuation region and $\Va=g$ in the stopping region we obtain as a suitable candidate to be the representing measure
\begin{equation}\label{eq:repMeasure}
\mu(dy)=
\begin{cases}
0,                            				&\text{ if $y<x^*$},\\
k\delta_{x^*}(dy),  	&\text{ if $y=x^*$},\\
(\alpha-{L})g(y)m(dy),  			&\text{ if $y>x^*$},
\end{cases}
\end{equation}

This approach was initiated by Salminen in 
\cite{salminen85} (see also \cite{mordecki-salminen}).
According to Salminen's approach,
once the excessive function is represented as an integral with respect to the Martin kernel $\Ma(x,y)$, 
\begin{equation}\label{eq:martin}
V_{\alpha}(x)=\int_{\I} \Ma(x,y)\kappa(dy)
\end{equation}
one has to find the representing measure $\kappa$. Martin and Green kernels are related by $\Ma(x,y)=\frac{\Ga(x,y)}{\Ga(x_0,y)}$, where $x_0$ is a reference point. Therefore, Riesz's representation of $\Va$ is related with the one in \eqref{eq:martin} by considering $\nu(dy)=\frac{\kappa(dy)}{\Ga(x_0,y)}$ and $\Ma(x,\ell)\kappa(\{\ell\})+\Ma(x,r)\kappa(\{r\})$ as the $\alpha$-harmonic function.

It is useful to observe that when the optimal stopping problem \eqref{eq:osp} is right-sided with optimal threshold $x^*$ it has a value function $\Va$ of the form
\begin{equation*}
 \Va(x)=
\begin{cases}
 \Ex{x}{ \ea{\hit{x^*}}} g(x^*), \quad &x<x^*,\\
 g(x),& x\geq x^*.
\end{cases}
\end{equation*}
Furthermore, $\Va(x)\geq g(x)$ for all $x\in \I$ and, in virtue of equation \eqref{eq:hitting}, we have
\begin{equation}
\label{eq:VaCall}
 \Va(x)=
\begin{cases}
 \frac{g(x^*)}{\psia(x^*)}\psia(x), \quad &x<x^*,\\
 g(x),& x\geq x^*.
\end{cases}
\end{equation}

The state space of the process can include or not the left endpoint $\ell$ and the right endpoint $r$.
In order to simplify, with a slight abuse of notation,
we write $[\ell,x]$, $[\ell,x)$, $[x,r]$, $(x,r]$ to denote 
respectively $\I\cap\{y\leq x\}$, $\I\cap\{y<x\}$, $\I\cap\{y\geq x\}$, $\I\cap\{y>x\}$.

 We say that the function $g\colon\I \mapsto \R$ satisfies the  \emph{right regularity condition} (RRC) 
 if there exist a point $x_1\in\I$ and a function $\gi\colon \I \to \R$ (not necessarily non-negative) such that $\gi(x)=g(x)$ for $x\geq x_1$ and 
\begin{equation}
\label{eq:invh}
 \gi(x)=\int_{\I} \Ga(x,y) \al \gi(y) m(dy)\quad (x\in\I).
\end{equation}
Proposition \ref{propInversion2} gives conditions in order to verify the inversion formula \eqref{eq:invh}. Informally speaking, the RRC is fulfilled by functions $g$ that satisfy all the local conditions --regularity conditions-- to belong to $\D_L$ for $x\geq x_1$, and does not increase as quick as $\psia$ does when approaching $r$ (in the case $r\notin \I$). Observe that if $\g$ satisfies the RRC for certain $x_1$ it also satisfies it for any greater value; and of course, if $\g$ itself satisfy \eqref{eq:invh} then it satisfies the RRC for all $x_1$ in $\I$.  To take full advantage of the following result it is desirable to find the least $x_1$ such that the RRC holds.

The main result follows.

\begin{theorem}
\label{mainteo} Consider a diffusion $X$ and a reward function $g$ that satisfies the RRC for some $x_1$. The optimal stopping problem is right-sided with optimal threshold $x^*\geq x_1$ if and only if:
\begin{equation}
\label{eq:defx*teo}
g(x^*)\geq \wa^{-1} \psia(x^*)\int_{(x^*,r]}\phia(y)\al g(y) m(dy),
\end{equation}
\begin{equation}
\label{eq:algmayor0}
 \al g(x)\geq 0, \quad x\in \I \colon x> x^*
\end{equation}
and 
\begin{equation}
\label{eq:mayorizq}
 \frac{g(x^*)}{\psia(x^*)}\psia(x)\geq g(x), \quad x\in \I \colon x<x^*.
\end{equation}

Furthermore, in the previous situation:
\begin{itemize}
\item Riesz's representation of the value function $\Va$ has representing measure as given in \eqref{eq:repMeasure} with 
$$k=\frac{g(x^*)-\wa^{-1} \psia(x^*)\int_{(x^*,r]}\phia(y)\al g(y) m(dy)}{\wa^{-1} \psia(x^*)\phia(x^*)},$$
while the $\alpha$-harmonic part vanishes;
\item if $x^*>x_1$ and the inequality \eqref{eq:defx*teo} is strict, then $x^*$ is the smallest number satisfying this strict inequality and \eqref{eq:algmayor0}, in particular 
\begin{equation}
\label{eq:otrax*}
g(x^*)\leq \wa^{-1} \psia(x^*)\int_{[x^*,r]}\phia(y)\al g(y) m(dy);
\end{equation}
which implies that $ k \leq \al g(x^*)m(\{x^*\})$;
\end{itemize}
\end{theorem}
\begin{remark}
From this theorem we obtain an algorithm to solve right-sided optimal stopping problems which works in most cases:
(i) Find the largest root $x^*$ of the equation 
\begin{equation}
\label{eq:defx*teoigual}
g(x^*)= \wa^{-1} \psia(x^*)\int_{(x^*,r]}\phia(y)\al g(y) m(dy);
\end{equation}
(ii) Verify $\al g(y)\geq 0$ for $x\geq x^*$;
(iii) Verify $g(x)\leq \frac{g(x^*)}{\psia(x^*)}\psia(x)$.
If these steps are fulfilled, the problem is right-sided with optimal threshold $x^*$. Observe that if $m(\{x^*\})=0$, then inequalities \eqref{eq:defx*teo} and \eqref{eq:otrax*} are equalities;

\end{remark}

\begin{proof}
We start by observing that if the problem is right-sided with threshold $x^*$ then \eqref{eq:algmayor0} holds. In general $\al g$ is non-negative in the stopping region (this can be seen with the help of the Dynkin's operator, see ex. 3.17 p. 310 in \cite{revuz}, see also equation (10.1.35) in \cite{oksendal}). Under the made assumption the value function $\Va$ is given by \eqref{eq:VaCall},
which implies \eqref{eq:mayorizq} since the value function dominates the reward. To finish the proof of the ``only-if'' part it remains to prove \eqref{eq:defx*teo}. Consider $\Wa\colon \I \mapsto \R$ defined by
\begin{equation*}
 \Wa(x):=\int_{(x^*,r]} \Ga(x,y) \al g(y) m(dy);
\end{equation*}
observe that $\Wa$ is $\alpha$-excessive in virtue of \eqref{eq:algmayor0} and Riesz's representation. Let $\Vat\colon \I \mapsto \R$ be defined by
\begin{equation*}
 \Vat(x):=\Wa(x)+k\Ga(x,x^*),
\end{equation*}
where $k$ is such that $\Vat(x^*)=g(x^*)$, i.e.
$k=\left(g(x^*)-\Wa(x^*)\right)/\Ga(x^*,x^*)$. Observe that, by \eqref{eq:Garepr},  $\Wa(x^*)$ is the right-hand side of \eqref{eq:defx*teo}. In fact, \eqref{eq:defx*teo} holds if and only if $k\geq 0$. By the definition of $\Vat$ and the representation \eqref{eq:Garepr} of $\Ga$ we get for $x\leq x^*$
\begin{equation*}
\Vat(x)=\frac{\psia(x)}{\psia(x^*)}\Vat(x^*)=\frac{\psia(x)}{\psia(x^*)}g(x^*)=\Va(x).
\end{equation*}
Let us compute $\Vat(x)-g(x)$ for $x \geq x^*$. In this region we have $g=\gi$, where $\gi$ is the extension given by the RRC. For $\gi$ can use the inversion formula \eqref{eq:invh}. Denoting by $\nu_{\gi}(dy)=\al \gi(y) m(dy)$ we have
\begin{equation*}
\aligned
\Vat(x)-g(x) &= k\Ga(x,x^*) -\int_{[\ell,x^*]} \Ga(x,y) \nu_{\gi}(dy)\\ 
&=k\wa^{-1} \phia(x)\psia(x^*) -\wa^{-1} \phia(x) \int_{[\ell,x^*]} \psia(y) \nu_{\gi}(dy)\\
&=\frac{\phia(x)}{\phia(x^*)} (\Vat(x^*)-g(x^*))=0,
\endaligned
\end{equation*}
because $\Vat(x^*)=g(x^*)$. So far, we have proved that $\Vat(x)=\Va(x)$ for all $x\in \I$. We are ready to prove that $k\geq 0$, based on the uniqueness of Riesz's decomposition: the $\alpha$-excessive function $\Wa$ has Riesz's representation given by its definition, and, if $k<0$ then 
$$\Wa(x)=-k \Ga(x,x^*) + \Va(x)$$
would give another Riesz's representation (the representing measure being $-k\delta_{\{x^*\}}(dx)+\mu(dx)$, where $\mu$ is the representing measure of $\Va$). An easy way of verifying that the measures are not the same is to observe that the former does not charge $\{x^*\}$, while the latter do.

To prove the ``if'' statement observe that, assuming \eqref{eq:defx*teo} \eqref{eq:algmayor0} and \eqref{eq:mayorizq}, function $\Vat$, already defined, is $\alpha$-excessive (by Riezs's representation, bearing in mind that $k\geq 0$) and dominates $g$. By Dynkin's characterization, the value function $\Va$ is the minimal $\alpha$-excessive function that dominates $g$. Therefore $\Vat \geq \Va$. Since $\Vat$ satisfies \eqref{eq:VaCall} (we have proved this in the first part of this proof), it follows that $\Vat$ is the expected reward associated with the hitting time of the set $[x^*,r]$, then 
$$\Vat(x) \leq \Va(x)=\sup_{\tau}\Ex{x}{\ea{\tau}g(X_{\tau})},$$
concluding that the problem is right-sided with threshold $x^*$.

The consideration about Riesz's representation of $\Va$ stated in the ``furthermore'' part are a direct consequence of the made proof.

To prove the minimality of $x^*$, 
suppose that there exists $x^{**}$ such that $x_1<x^{**}<x^*$ 
satisfying the strict inequality in \eqref{eq:defx*teo} and \eqref{eq:algmayor0}. 
Let us check $\Va(x^{**})-\g(x^{**})<0$, 
in contradiction with the fact that $\Va$ is a majorant of $g$. 
Considering the extension $\gi$ given by the RRC and denoting $\nu_{\gi}(dy)=\al \gi(y) m(dy)$ we have
\begin{equation*}
\Va(x^{**})-g(x^{**}) = -\int_{[\ell,x^*]}\Ga(x^{**},y)\nu_{\gi}(dy)+k \Ga(x^{**},x^*)=s_1+s_2+s_3+s_4,
\end{equation*}
where
\begin{align*}
s_1&=-\int_{[\ell,x^{**}]}\Ga(x^{**},y)\nu_{\gi}(dy),\\
s_2&=-\wa^{-1} \psia(x^{**}) \int_{(x^{**},x^*]} \phia(y) \nu_{\gi}(dy),\\
s_3&=\frac{\psia(x^{**})}{\psia(x^*)}\frac{\phia(x^*)}{\phia(x^{**})} \int_{[\ell,x^{**}]}\Ga(x^{**},y)\nu_{\gi}(dy),\\
s_4&= \frac{\psia(x^{**})}{\psia(x^*)} \int_{(x^{**},x^*]}\Ga(x^*,y)\nu_{\gi}(dy).
\end{align*}
To check that $s_3+s_4=k \Ga(x^{**},x^*)$, use
\begin{equation*}
 k=\frac{1}{\Ga(x^*,x^*)}\int_{[\ell,x^*]}\Ga(x^*,y)\nu_{\gi}(dy), \quad \mbox{and}\quad \frac{\Ga(x^{**},x^*)}{\Ga(x^*,x^*)}=\frac{\psia(x^{**})}{\psia(x^*)}.
\end{equation*}
%
Finally, observe that
\begin{equation*}
s_1+s_3= \left(\frac{\psia(x^{**})}{\psia(x^*)}\frac{\phia(x^*)}{\phia(x^{**})}-1\right) \int_{[\ell,x^{**}]}\Ga(x^{**},y)\nu_{\gi}(dy)< 0, 
\end{equation*}
because the first factor is negative and the second one positive, 
by the assumption about $x^{**}$, and
\begin{equation*} 
s_4+s_2=\frac{\wa^{-1} \psia(x^{**})}{\psia(x^*)} \int_{(x^{**},x^*]}\left(\psia(y)\phia(x^*)-\psia(x^*)\phia(y)\right)\nu_{\gi}(dy)<0,
\end{equation*}
because the measure $\nu_{\gi}(dy)$ is positive and the integrand non-positive (it is increasing and vanishes in $y=x^*$). We have obtained that 
$\Va(x^{**})-g(x^{**})=s_1+s_2+s_3+s_4 <0$, concluding the proof.
\end{proof}

The previous theorem gives necessary and sufficient conditions for the problem \eqref{eq:osp} to be right-sided under the RRC.
The following result gives simpler sufficient conditions for the problem to be right-sided.

\begin{theorem}
Consider a diffusion $X$ and a reward function $g$ such that \eqref{eq:invh} is fulfilled for $\gi=g$.
Suppose that there exists a root $c\in (\ell,r)$ of the equation $\al g(x)=0$,
such that $\al g(x)<0$ if $x<c$ and $\al g(x)>0$ if $x>c$,
 and that  $\int_{\I} \psia(y) \al g(y) m(dy)\in (0,\infty]$.
Then the optimal stopping problem \eqref{eq:osp} is right-sided, with optimal threshold 
\begin{equation}
\label{eq:xstarsuficiente}
x^*=\min\{x\colon b(x)\geq 0\} ,
\end{equation}
where
\begin{equation*}
b(x)=\int_{[\ell,x]} \psia(y) \al g(y) m(dy)\quad (x\in\I).
\end{equation*}
\end{theorem}
\begin{proof}
The idea is to apply Theorem \ref{mainteo}, 
with $x^*$ defined in \eqref{eq:xstarsuficiente}. By the assumptions on $\al g$ and the fact that $m(dy)$ is strictly positive in any open set,
 we obtain that the function $b(x)$ is decreasing in $[\ell,c)$ and increasing in $(c,r)$. 
 Moreover $b(x)<0$ if $\ell<x\leq c$. 
 Since $b$ is right continuous and increasing in $(c,r)$, 
 the set $\{x\colon b(x)\geq 0\}=[x^*,r)$ 
 with $x^*> c$. 
 Observe that, by \eqref{eq:invh} and \eqref{eq:Garepr} we get
\begin{equation*}
 g(x)=\wa^{-1}\phia(x)b(x)+\int_{(x,r]}\Ga(x,y)\al g(y)m(dy),
\end{equation*}
and $b(x^*)\geq 0$ is equivalent to \eqref{eq:defx*teo}. Since $x^*\geq c$ we have $\al g(y)>0$ for $x>x^*$. 
It only remains to verify \eqref{eq:mayorizq}. 
By definition of $x^*$ there exists a signed measure $\sigma_{\ell}(dy)$ whose support is contained in 
$[\ell,x^*]$, and $\sigma_{\ell}(dy)=\al g(y) m(dy)$ for $y<x^*$ and such that
\begin{equation*}
\int_{[\ell,x^*]}\psia(y)\sigma_{\ell}(dy)=0.
\end{equation*}
Furthermore $\sigma_r(dy)=\al g(y)m(dy)-\sigma_{\ell}(dy)$ is a positive measure supported in $[x^*,r]$. 
Using the inversion formula for $g$ and \eqref{eq:Garepr}, we have for $x<x^*$
\begin{equation*}
 g(x)-\frac{\psia(x)}{\psia(x^*)}g(x^*)=\int_{[\ell,x^*]} \Ga(x,y) \sigma_{\ell}(dy)\leq \frac{\Ga(x,c)}{\psia(c)} \int_{[\ell,x^*]}\psia(y)  \sigma_{\ell}(dy) =0,
\end{equation*}
where the inequality follows from the following facts:
if $y<c$ then $\sigma_{\ell}(dy)\leq 0$ and
\begin{equation*}
\psia(y) \frac{\Ga(x,c)}{\psia(c)} \leq \Ga(x,y),
\end{equation*}
while if $y>c$ then $\sigma_{\ell}(dy)\geq 0$
\begin{equation*}
\psia(y) \frac{\Ga(x,c)}{\psia(c)} \geq \Ga(x,y).
\end{equation*}
We can now apply Theorem \ref{mainteo} completing the proof.
\end{proof}

\subsection{On the right regularity condition (RRC)}
In order to apply the previous results it is necessary to verify the inversion formula \eqref{eq:invh}. 
As we have seen in the preliminaries, 
if $f\in \D_L$ we have $\Ra \al f = f$, and if equation \eqref{eq:Ra=Ga} holds for $\al f$, 
we have \eqref{eq:invh}. 
This is the content of the following result.
\begin{lemma}
 Assume $f\in \D_L$, and 
\begin{equation*}
 \int_{\I} \Ga(x,y) |\al f(y)| m(dy) < \infty \quad \text{for all $x\in \I$}.
\end{equation*}
Then $f$ satisfies equation \eqref{eq:invh}.
\end{lemma}
The conditions of the previous lemma are very restrictive in order to solve concrete problems, 
as reward functions typically satisfy $\lim_{x\to r}g(x)=\infty$. 
The following result extends the previous one to unbounded functions.

\begin{proposition} \label{propInversion2} Consider the case $r\notin \I$.
Suppose $u\colon\I\mapsto \R$ is such that $Lu(x)$ in \eqref{eq:difop} can be defined for all $x\in\I$. 
Assume
\begin{equation}\label{eq:alIntegrable}
 \int_{\I}\Ga(x,y) |\al u(y)| m(dy)<\infty,
\end{equation}
and
\begin{equation} \label{eq:gOverPsi}
 \lim_{z\to r^-} \frac{u(z)}{\psia(z)}=0.
\end{equation}
Suppose also that for each $y \in \I$ there exist a function $u_y \in \D_L$ such that $u_y(x)=u(x)$ for $x\leq y$. 
Then $u$ satisfies \eqref{eq:invh}.
\end{proposition}
\begin{proof}
By \eqref{eq:Ra=Ga} we have $\int_{\I} \Ga(x,y) \al u(y) m(dy)=\Ra \al u(x)$. 
Consider a strictly increasing sequence $r_n\to r\ (n\to\infty)$ 
and denote by $u_n$ the function $u_{r_{n+1}} \in \D_L$ of the hypothesis. 
By the continuity of the sample paths, by our assumptions on the right boundary $r$, we have $\hit{r_n}\to \infty\quad (n\to\infty)$. 
Applying formula \eqref{eq:dynkinFormula} to $u_n$ and the stopping time $\hit{r_n}$ we obtain, 
for $x<r_n$,
\begin{equation*}
u_n(x)=\Ex{x}{\int_0^{\hit{r_n}} \ea{t} \al u_n(X_t) dt}+ \Ex{x}{\ea{\hit{r_n}}}u_n(r_n),
\end{equation*}
using $u_n(x)=u(x)$ and $\al u(x)=\al u_n(x)$ for $x<r_{n+1}$ we have
\begin{equation*}
u(x)=\Ex{x}{\int_0^{\hit{r_n}} \ea{t} \al u(X_t) dt}+ \Ex{x}{\ea{\hit{r_n}}}u(r_n).
\end{equation*}
Taking limits when $n\to \infty$, by  \eqref{eq:hitting} and \eqref{eq:gOverPsi} we have
\begin{equation*}
 \Ex{x}{ \ea{\hit{r_n}}}u(r_n) = \frac{\psia(x)}{\psia(r_n)}u(r_n)\to 0.
\end{equation*}
To compute the limit of the first term above we use dominated convergence theorem and \eqref{eq:alIntegrable}.
The result is
\begin{align*}
 u(x)&=\int_0^\infty \Ex{x}{ \ea{t} \al u(X_t) }dt
= \int_{\I} \Ga(x,y) \al u(y) m(dy)
\end{align*}
concluding the proof.
\end{proof}

\section{On the principle of smooth fit}\label{section:sf}

The principle of \emph{smooth fit}  (\sfp) holds when condition $V'(x^*)=g'(x^*)$ is satisfied, 
being a helpful tool to find candidate solutions to optimal stopping problems.
In \cite{salminen85} Salminen proposes an alternative version of this principle, considering derivatives with respect to the scale function. 
We say that there is \emph{scale smooth fit} (\ssfp) when the value function has derivative at $x^*$ 
with respect to the scale function. 
Note that if $g$ also has derivative with respect to the scale function they coincide, since $g=\Va$ in $[x^*,r]$. 
In \cite{peskir07} Peskir presents two interesting examples: one of them consists on the optimal stopping problem of a regular diffusion with a differentiable payoff function in which the principle of \sfp\ does not hold, but the alternative principle of \ssfp\ does; while in the other the principle of \sfp\ holds but the principle of \ssfp\ fails. Later, Samee \cite{samee2010principle} 
analysed the validity of the principle of smooth fit for killed diffusions and introduced other alternative principles considering derivatives of $\frac{g}{\psia}$ and $\frac{g}{\phia}$ with respect to the scale function $s$. See also the paper by Jacka \cite{jacka1993local} for a study of the principle of smooth fit related to the Snell envelope.

We now  analyse the relation between Riez's representation of $\Va$, stated in the previous section, and the principle of smooth fit. We start by proving that  $k=\nu(\{x^*\})=0$ in \eqref{eq:repMeasure} implies that the reward function has derivatives with respect to the function $\psia$. Then we follow by stating some corollary results.

\begin{theorem} \label{teosf}
Given a diffusion $X$ and a reward function $g$, if the value function associated with the problem \eqref{eq:osp} satisfies 
\begin{equation*}
 \Va(x)=\int_{(x^*,r]}\Ga(x,y) \al g(y) m(dy),
\end{equation*}
for $x^*\in (\ell,r)$, then $\Va$ is differentiable at $x^*$ with respect to $\psia$. 
\end{theorem}
\begin{proof}
For $x\leq x^*$
\begin{equation*}
\Va(x)=\wa^{-1} \psia(x) \int_{(x^*,r]}\phia(y) \nu_g(dy), 
\end{equation*}
and the left derivative of $\Va$ with respect to $\psia$ in $x^*$ is
\begin{equation*}
\frac{\partial \Va^-}{\partial \psia}(x^*)=\wa^{-1} \int_{(x^*,r]}\phia(y) \nu_g(dy). 
\end{equation*}
For $x>x^*$
\begin{equation*}
\Va(x)=\phia(x) \wa^{-1}\int_{(x^*,x)}\psia(y)\nu_g(dy)+\psia(x) \wa^{-1}\int_{[x,r]}\phia(y)\nu_g(dy).
\end{equation*}
Computing the difference between $\Va(x)$ and $\Va(x^*)$ we obtain
\begin{align*} \Va(x)-\Va(x^*)& =\wa^{-1} \left(\psia(x)-\psia(x^*)\right) \int_{[x,r]}\phia(y)\nu_g(dy)\\
&\qquad +\wa^{-1}\int_{(x^*,x)}
\left(\phia(x) \psia(y)-\psia(x^*)\phia(y)\right) \nu_g(dy).
\end{align*}
Then
\begin{align*}
 \frac{\partial \Va^+}{\partial \psia}(x^*)&= \lim_{x \to {x^*}^+} \frac{\Va(x)-\Va(x^*)}{\psia(x)-\psia(x^*)} \\
&= \wa^{-1} \lim_{x \to {x^*}^+} \int_{[x,r]}\phia(y)\nu_g(dy) \\
& \qquad + \wa^{-1} \lim_{x \to {x^*}^+} \frac{\int_{(x^*,x)}\phia(x)\psia(y)-\psia(x^*)\phia(y)\nu_g(dy)}{\psia(x)-\psia(x^*)}.
\end{align*}
If the last limit vanishes, we obtain that the right derivative
exists, and
\begin{equation*}
\frac{\partial \Va^+}{\partial \psia}(x^*)=\frac{\partial \Va^-}{\partial \psia}(x^*)=\wa^{-1} \int_{(x,r]}\phia(y)\nu_g(dy).
\end{equation*}
This means that we have to prove 
\begin{equation}\label{eq:integral0}
\lim_{x \to {x^*}^+} \int_{(x^*,x)}\frac{\phia(x)\psia(y)-\psia(x^*)\phia(y)}{\psia(x)-\psia(x^*)}\nu_g(dy)=0.
\end{equation}
Denoting by $f(y)$ the numerator of the integrand in \eqref{eq:integral0}, observe that
\begin{equation*}
 f(y)=\phia(x)(\psia(y)-\psia(x^*))+\psia(x^*)(\phia(x)-\phia(y)).
\end{equation*}
For the first term, we have (observe that $x^* < y < x$)
\begin{equation*}
0\leq \phia(x)(\psia(y)-\psia(x^*))\leq \phia(x)(\psia(x)-\psia(x^*)),
\end{equation*}
while for the second
\begin{equation*}
0\geq \psia(x^*)(\phia(x)-\phia(y)) \geq \psia(x^*)(\phia(x)-\phia(x^*)).
\end{equation*}
We conclude that
\begin{equation*}
\psia(x^*)(\phia(x)-\phia(x^*))\leq f(y) \leq \phia(x)(\psia(x)-\psia(x^*)).
\end{equation*}
Dividing by $\psia(x)-\psia(x^*)$ we see that the integrand has a lower bound $b(x)$ given by
\begin{equation*}
b(x):=\psia(x^*)\frac{\phia(x)-\phia(x^*)}{\psia(x)-\psia(x^*)} = \psia(x^*)\frac{\phia(x)-\phia(x^*)}{s(x)-s(x^*)} \frac{s(x)-s(x^*)}{\psia(x)-\psia(x^*)},
\end{equation*}
while $\phia(x)$ is an upper bound. 
We obtain the integral in \eqref{eq:integral0} satisfies
\begin{equation*}
 b(x)\nu_g(x^*,x)\leq  \int_{(x^*,x)}\frac{\phia(x)\psia(y)-\psia(x^*)\phia(y)}{\psia(x)-\psia(x^*)}\nu_g(dy) \leq \phia(x)\nu_g(x^*,x).
\end{equation*}
Taking limits when $x \to {x^*}^+$ we obtain $\phia(x)\to \phia(x^*)$, $\nu_g(x,x^*)\to 0$, and
\begin{equation*}
 \lim_{x\to {x^*}^+}b(x)=\psia(x^*)\frac{\partial \phia^+}{\partial s}(x^*)\left/ \frac{\partial \psia^+}{\partial s}(x^*)\right.,
\end{equation*}
concluding that \eqref{eq:integral0} holds.
\end{proof}

As we have seen in section \ref{section:main},
if the speed measure does not charge $x^*$ neither does the representing measure. This means that if representation \eqref{eq:repMeasure} holds, then $m(\{x^*\})=0$ is enough to guarantee the differentiability of $\Va$ with respect to $\psia$. 
We also have the following result.
\begin{corollary}
\label{scalesf}
Assume the conditions of Theorem \ref{teosf}.  If the speed measure does not charge $x^*$ there is scale smooth fit.
\end{corollary}
\begin{proof}
By the previous theorem we know that $\Va$ is differentiable with respect to $\psia$. 
Condition $m(\{x^*\})=0$ implies that $\psia$ has derivative with respect to the scale function. We conclude
\begin{equation*}
 \frac{\partial \Va}{\partial s}(x^*)=\frac{\partial \Va}{\partial \psia}(x^*) \left/\frac{\partial \psia}{\partial s}(x^*)\right.
\end{equation*}
\end{proof}
The previous result, under the additional assumption that $\psia$ and $\phia$ are differentiable with respect to $s$, could be derived from Corollary 3.7 in \cite{salminen85}. This result states that the representing measure of $\Va$ does not charge $x^*$ if and only if $\Va$ is differentiable with respect to $s$. Also Theorem \ref{teosf} can be derived from the mentioned result under the additional assumption by using the chain rule.

As a consequence of the previous results, we obtain, by using the chain rule, conditions under which the principle of \sfp\ holds. 
\begin{corollary}
 \label{clasicsf} 
Assume that $g$ is differentiable at $x^*$. Under the conditions of Theorem \ref{teosf}, if $\psia$ is differentiable at $x^*$ and $\psia'(x^*)\neq 0$ (or under the conditions of Corollary \ref{scalesf}, if $s$ is differentiable at $x^*$ and $s'(x^*)\neq 0$) then the principle of \sfp\ holds.

\end{corollary}

The previous result is closely related with Theorem 2.3 in \cite{peskir07}, which states that, in the non-discounted problem, there is smooth fit if the reward and the scale function are differentiable at $x^*$. Theorem 2.3 in \cite{samee2010principle} ensures the validity of the smooth fit principle for the discounted problem under the assumption that $g$, $\psia$ and $\phia$ are differentiable at $x^*$. It should be noticed that these results are valid in general, not only in one-sided problems.

\section{Examples}\label{section:examples}

In this section we show how to solve some optimal stopping problems using the previous results. 
We present two classical examples (American and Russian options), 
and also include some new examples in which the smooth fit principle is not useful to find the solution.

\subsection{American call options}

Consider  a geometric Brownian motion given by 
$
X_t=x\exp(\sigma W_t+(\mu-\sigma^2/2)t)
$,
where $\{W_t\}$ is a standard Brownian motion, $\mu\in \R$ and $\sigma^2>0$.
The state space is $I=(0,\infty)$.  We refer to \cite{borodin}, p. 132 for the basic characteristics of this process.
 The infinitesimal generator is  $Lf=\frac{1}{2}\sigma^2x^2 f''+ \mu x f',$ with domain 
 \begin{equation*}
\D_L=\{f: f,\ Lf \in \CC_b(\I)\}. 
\end{equation*} 
Consider the payoff function $g(x)=(x-K)^+\ (x\in\R)$,
where $K$ is a positive constant, and a positive discount factor $\alpha$ satisfying $\alpha>\mu$. 
The reward function $g$ satisfies the RRC for $x_1=K$: 
it is enough to consider $\gi\in\CC^2$, bounded in $(0,K)$ and such that $\gi(x)=x-K$ for $x\geq K$. 
This function $\gi$ satisfies the inversion formula \eqref{eq:invh}
as a consequence of Proposition \ref{propInversion2}. 
Observe that equation \eqref{eq:alIntegrable} holds. Equation \eqref{eq:gOverPsi} is in this case
\begin{equation*}
 \lim_{z\to \infty} \frac{\gi(z)}{\psia(z)} = \lim_{z\to \infty} z^{1-\gamma_1},
\end{equation*}
with 
\begin{equation*}
\psia(x)=x^{\gamma_1}, \text{ and }\gamma_1=\frac{1}{2}-\frac{\mu}{\sigma^2}{+\sqrt{\left(\frac{1}{2}-\frac{\mu}{\sigma^2}\right)^2+ \frac{2\alpha}{\sigma^2}}}.
\end{equation*}
The last limit vanishes if $1-\gamma_1<0$, which is equivalent to $\mu < \alpha$.
To find $x^*$ we solve equation \eqref{eq:defx*teoigual}. 
After computations, we find 
\begin{equation*}
x^*=K \left(\frac{\gamma_1}{\gamma_1-1}\right).
\end{equation*}
It is not difficult to verify \eqref{eq:algmayor0} and \eqref{eq:mayorizq} in order to apply Theorem \ref{mainteo}. We conclude that the problem is right-sided with optimal threshold $x^*$. Observe that the hypotheses of Theorem \ref{teosf} and corollaries \ref{scalesf} and \ref{clasicsf} are fulfilled. In consequence all the variants of smooth fit principle hold. This problem was solved by Merton in \cite{merton}.

\subsection{Russian Options}

The Russian Option was introduced by Shepp and Shiryaev in 1993 in  \cite{shepp93}, where the option pricing problem is solved by reduction to an optimal stopping problem of a two-dimensional Markov process. 
Later, in \cite{shepp94}, the authors give an alternative approach to the same problem solving a one-dimensional optimal stopping problem. In 2000, Salminen \cite{salminen00RussianOptions}, making use of a generalization of 
L\'evy's theorem for Brownian motion with drift, shortened the derivation of the valuation formula in \cite{shepp94} and solved the optimal stopping problem related. 

Consider $\alpha>0$, $r>0$ and $\sigma>0$. Let $\{X_t\}$ be a Brownian motion on $I=[0,\infty)$,  with drift $-\delta < 0$, where $\delta=\frac{r+\sigma^2/2}{\sigma} $ and reflected at $0$.
In \cite{salminen00RussianOptions} it is shown that the optimal stopping problem to be solved has underlying process $\{X_t\}$
and reward function $g(x)=e^{\sigma x}$. For the basic characteristics of the process we refer to  \cite{borodin}, p. 129.
The infinitesimal generator is $Lf(x)=f''(x)/2-\delta f'(x)$ for $x>0$ and $Lf(0)=\lim_{x\to 0^+} Lf(x)$,
 with domain 
\begin{equation*}
\D_L=\{f: f,\ Lf \in \CC_b(\I), \lim_{x\to 0^+}f'(x)=0\}. 
\end{equation*} 
The payoff function $g(x)$ satisfies the RRC for every $x_1>0$: for $x_1>0$ it is easy to find a function $\gi$ with continuous second derivatives such that $\gi=g$ in $[x_1,\infty)$ and such that the right derivative at 0 is 0. By the application of Proposition \ref{propInversion2} we obtain that $\gi$ satisfies the inversion formula \eqref{eq:invh}, then the RRC holds. We obtain
\begin{equation*}
\al g(x)=(\alpha - \sigma^ 2/2 + \delta \sigma)e^{\sigma x} = (\alpha + r)e^{\sigma x}{>0}. 
\end{equation*}
In order to apply Theorem \ref{mainteo} we solve equation \eqref{eq:defx*teoigual} which in this case is
\begin{equation*}
e^{\sigma x}=\frac{1}{\gamma-\delta}\left(\frac{\gamma - \delta}{2 \gamma} e^{(\gamma+\delta) x}+\frac{\gamma + \delta}{2 \gamma} e^{-(\gamma-\delta) x} \right)\int_{x}^\infty 2 (\alpha+r) e^{(-\gamma-\delta+\sigma) y}dy, 
\end{equation*}
with $\gamma=\sqrt{2 \alpha + \delta^2}$, obtaining (observe that  $-\gamma-\delta+\sigma<0$)
\begin{equation*}
x^*=\frac{1}{2\gamma}\ln{\left(\left(\frac{\gamma+\delta}{\gamma-\delta}\right)\left(\frac{\gamma-\delta+\sigma}{\gamma+\delta-\sigma}\right)\right)}. 
\end{equation*}
It is easy to verify conditions \eqref{eq:algmayor0} and \eqref{eq:mayorizq}, 
to obtain, by application of Theorem \ref{mainteo}, that the problem is right-sided with threshold $x^*$, as proved in \cite{shepp94}.

\subsection{Skew Brownian motion}
We consider a Brownian motion skew at zero. This process behaves like a standard Brownian motion outside the origin, 
but has an asymmetric behaviour when hitting $x=0$, modeling a permeable barrier. 
The behaviour at $x=0$ is regulated by a parameter $\beta\in (0,1)$,  
known as the \emph{skewness parameter}.
The state space of this process is $\I=\R$. For details on this process and its basic characteristics we refer to see \cite{borodin}, p. 126 or \cite{lejay}. 
The infinitesimal generator is $Lf(x)=f''(x)/2$ if $x\neq 0$ and $Lf(0)=\lim_{x\to 0}Lf(x)$, with domain 
\begin{equation*} 
\D_L=\{f: f,\ Lf \in \CC_b(\I),\ \beta f'(0^+)=(1-\beta)f'(0^-)\}. 
\end{equation*}

Consider the payoff function $g(x)=x^+$. Function $g$ satisfies the RRC for $x_1=0$: to see this it is necessary to construct $\gi$ such that $\gi=\g$ in $[0,\infty)$ $\gi$ with second derivative bounded in $(-\infty,0)$ and such that $\gi'(0^-)=\beta/(1-\beta)$ (so that $\gi$ satisfies the local conditions to belong to $\D_L$). Applying Proposition \ref{propInversion2} it can be concluded that $\gi$ satisfies \eqref{eq:invh}, so the RRC holds. We have $\al g(x)= \alpha x\  (x\geq0).$
Equation \eqref{eq:defx*teoigual} is in this case
\begin{equation*}
x^*=\frac{1}{\sqrt{2\alpha}}\left(\frac{1-2\beta}{\beta}\sinh(\sqrt{2\alpha}\ x^*)+e^{\sqrt{2\alpha}\ x^*}\right)\int_{x^*}^\infty e^{-\sqrt{2\alpha}\ t}\alpha t\ 2\beta dt 
\end{equation*}
or equivalently
\begin{equation} \label{eq:xxskew} x^*=\frac{1}{2\sqrt{2\alpha}}\left((2\beta-1)e^{-\sqrt{2\alpha}\ x^*}(\sqrt{2\alpha}\ x^*+1)+\sqrt{2\alpha}\ x^*+1\right).
\end{equation}
In general, this equation can not be solved analytically. 
If we consider the particular case $\beta=\frac{1}{2}$, in which the process is the ordinary Brownian motion, we obtain the known result
$x^*=\frac{1}{\sqrt{2\alpha}}$ (see \cite{taylor}). 
Consider a particular case, in which $\alpha=1$ and $\beta=0.9$. Solving numerically equation \eqref{eq:xxskew} we obtain
$
x^*\simeq 0.82575.
$
Checking \eqref{eq:algmayor0} and \eqref{eq:mayorizq} we conclude
that the problem is right-sided with optimal threshold $x^*$.

\subsubsection{An example without smooth fit}
Consider again the Skew Brownian motion with parameter $\beta=1/3$, a payoff function 
$g(x)=(x+1)^+$ and a discount $\alpha=1/8$.
We have $\al g(x)=\alpha (x+1)\ (x\geq0)$.
Observe that $x^*=0$ is a solution of \eqref{eq:defx*teo} (with equality). 
It is easy to see that the assumptions of Theorem \ref{mainteo} are fulfilled. 
We conclude that the problem is right-sided with threshold $x^*=0$. 
Moreover, we know
\begin{equation*}
\Va(x)=
\begin{cases}
x+1, \quad & x\geq 0, \\
\psi_\alpha(x) ,\quad  & x\leq 0. 
\end{cases}
\end{equation*}
where
\begin{equation*}
\psia(x)=
\begin{cases}
e^{\sqrt{2\alpha}\ x}, & x\leq 0, \\
\frac{1-2\beta}{\beta}\sinh(x \sqrt{2\alpha})+e^{\sqrt{2\alpha}\ x},\quad & x\geq 0.
\end{cases}
\end{equation*}

Unlike the previous examples, the value function $\Va$ is not differentiable at $x^*$. 
As we see in Figure \ref{fig:skewNoFit}, the graph of $\Va$ shows an angle at $x=0$. 
By application of Theorem \ref{teosf} and Corollary \ref{scalesf} we conclude that $\Va$ is differentiable with respect to $\psia$ and \ssfp\ hold.

An example considering a regular diffusion with non-differentiable scale function, in which the the SF fails to hold, was provided for the first time by Peskir \cite{peskir07}.
\begin{figure}
\begin{center}
\resizebox*{8 cm}{!}{\includegraphics{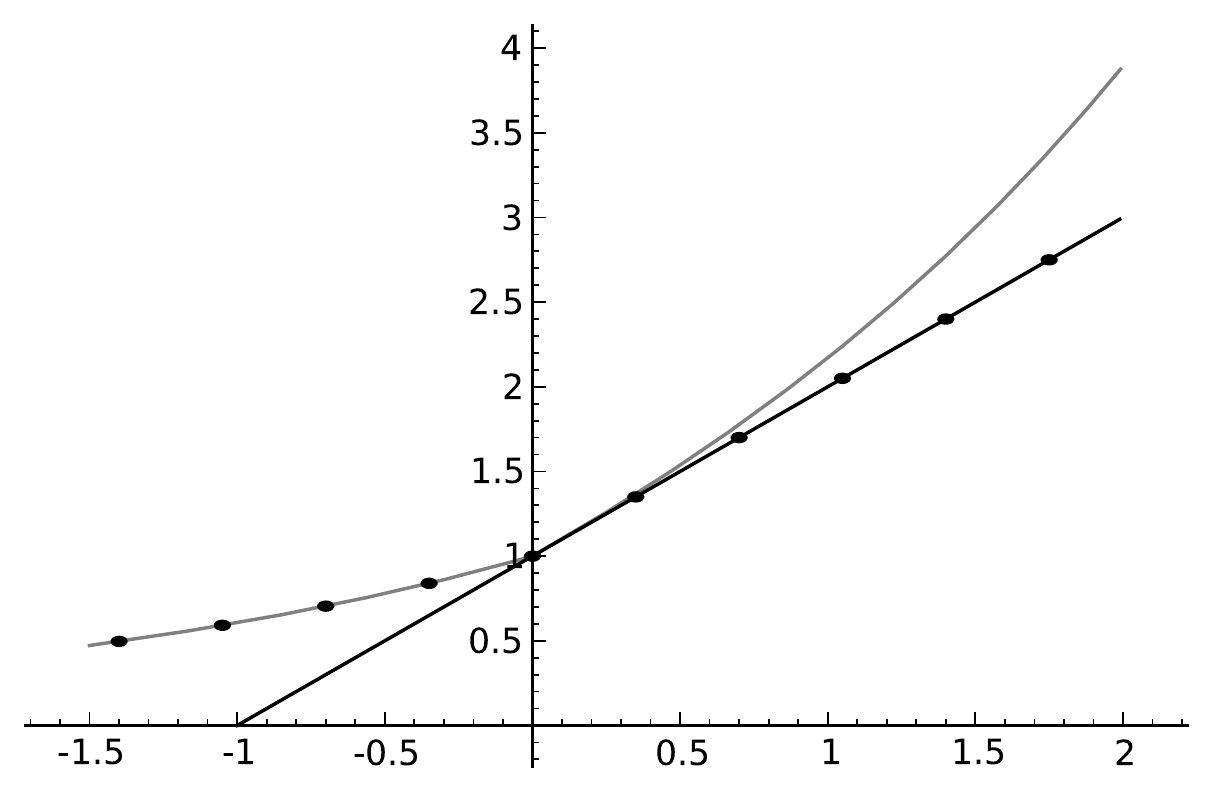}}%
\end{center}
\caption{\label{fig:skewNoFit} 
Solution of the OSP for the skew Brownian with $\alpha=1/8$, $\beta=1/3$ and reward function $g=(x+1)^+$. 
Black graph corresponds to $g$, while the gray one corresponds to $\psia$. 
$\Va$ is the dotted line.
}
\end{figure}
\subsection{Sticky Brownian Motion}
Consider a Brownian motion, sticky at 0. We refer to \cite{borodin} p. 123 for the basic characteristics of this process. 
It behaves like a  Brownian motion out of 0, but spends a positive time at $x=0$; this time depends on a positive parameter that we assume to be $1$ in this example.
The state space of this process is $\I=\R$. The scale function is $s(x)=x$ and the speed measure is $m(dx)=2dx+2\delta_{\{0\}}(dx)$. 
The infinitesimal generator is $Lf(x)=f''(x)/2$ when $x\neq 0$, and $Lf(0)=\lim_{x\to 0}Lf(x)$,
with domain 
\begin{equation*} 
\D_L=\{f: f,\ Lf \in \CC_b(\I),\ f''(0^+)=f'(0^+)-f'(0^-)\}. 
\end{equation*}
Consider the reward function $g(x)=(x+1)^+$, that satisfies the RRC for $x_1=-1$ (it can be seen with the same arguments considered so far). 
We discuss the solution of the optimal stopping problem in terms of the discount factor, 
in particular we are interested in finding values of $\alpha$ such that the optimal threshold is the sticky point. 
We use \eqref{eq:defx*teoigual} in a different way: we fix $x=0$ and solve the resulting equation in $\alpha$. 
We obtain
\begin{equation*}
 1=\wa^{-1} \int_{(0,\infty)} e^{-\sqrt{2\alpha}\ y}\alpha(y+1) 2dy
\end{equation*}
and we find that $\alpha_1=\frac{(-1+\sqrt{5})^2}{8}\simeq 0.19$ is the unique solution. 
It can be seen, by application of Theorem \ref{mainteo}, that if $\alpha=\alpha_1$ the problem is right-sided with threshold $x^*=0$. In this case the representing measure of $\Va$ does not charge $x^*$ despite the speed measure does. 
Furthermore, Theorem \ref{teosf} can be applied to conclude that $\Va$ is differentiable with respect to $\psia$. 
It also should be noticed that both SF and SSF fail to hold in this case. This was expectable because the sufficient conditions given in \cite{samee2010principle} and \cite{salminen85} for the different types of smooth fit are not fulfilled. 
Another interesting thing to remark is that $d(\Va/\phia)/ds$ (and also $d(\Va/\psia)/ds$) exists at 0, which is part of the conclusion of Theorem 2.1 in \cite{samee2010principle}, despite this result is not applicable in this case. This last fact seems to be related with the existence of $d\Va / d\psia$

Another approach to obtain $x^*=0$ is when the strict inequality holds in \eqref{eq:defx*teo} and also \eqref{eq:otrax*} holds. We solve (in $\alpha$)  equation \eqref{eq:otrax*} (with equality), which is
\begin{equation}
\label{eq:alpha2}
g(0)=\wa^{-1}\psia(0)\int_{[0,\infty)}\psia(y)\al g(y) m(dy).
\end{equation}
Since the measure $m(dy)$ has an atom at $y=0$, the solution of the previous equation differs from $\alpha_1$. 
Solving \eqref{eq:alpha2} we find the root $\alpha_2=1/2$. It is easy to see that for $\alpha \in [\alpha_1,\alpha_2]$, the minimal $x$ satisfying the inequality \eqref{eq:defx*teo} is 0. Theorem \ref{mainteo} can be applied to conclude that the problem is right-sided with threshold $x^*=0$.
For $\alpha \in (\alpha_1,\alpha_2]$  we cannot apply any of the results of section \ref{section:sf}, and in fact, for $\alpha\in (\alpha_1,\alpha_2)$ none of the smooth fit principles is fulfilled. 
With $\alpha=\alpha_2$ there is \sfp\ (and also \ssfp, since $s(x)=x$), but this is not a consequence of \eqref{eq:alpha2}, 
it is due to the particular reward function. 
This example shows that the theorems on smooth fit in section \ref{section:sf}
only gives sufficient conditions.
In table \ref{table} we summarize the information about the solution of the optimal stopping problem. We also give, in Figure \ref{fig:sticky}, some graphics showing the solution for different values of $\alpha$.

\begin{table}
\label{table} SBM: solution of the OSP depending on $\alpha$\\
{\begin{tabular}{@{}lccccccc}
\hline
$\alpha$ & $x^*$ & \eqref{eq:defx*teo} & \eqref{eq:otrax*} & \sfp\ \& \ssfp &$\exists  \ d\Va / d \psia$ & $\exists \ d(\Va / \phia)/ds$ & Fig.\\
\hline
$\alpha\in(0,\alpha_1)$ & $x^*>0$ & = & = & yes & yes & yes &\ref{fig:menor}\\
$\alpha=\alpha_1$ & $x^*=0$ & = & $<$ & no & yes & yes&\ref{fig:a1}\\
$\alpha\in(\alpha_1,\alpha_2)$& $x^*=0$ & $>$ & $<$ & no & no & no&\ref{fig:medio}\\
$\alpha=\alpha_2$ & $x^*=0$ & $>$ & = & yes & no & no&\ref{fig:a2}\\
$\alpha\in (\alpha_2,+\infty)$ & $x^*<0$ & = & = & yes & yes & yes&\ref{fig:mayor}\\
\hline
\end{tabular}}
\end{table}
\begin{figure}
\begin{center} 
\subfigure[\label{fig:menor}$\alpha=0.1$]{
\resizebox*{6.9cm}{!}{\includegraphics{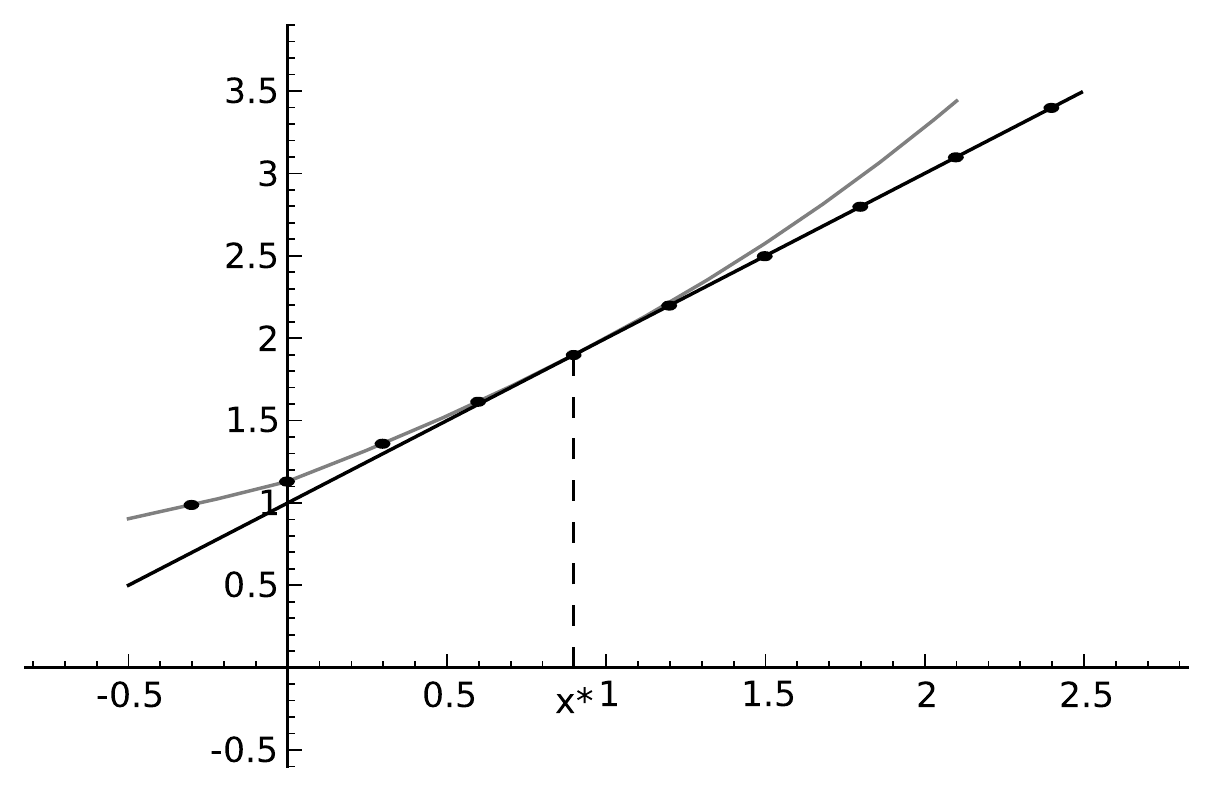}}}%
\subfigure[\label{fig:a1}$\alpha=\alpha_1\simeq 0.19$]{
\resizebox*{6.9cm}{!}{\includegraphics{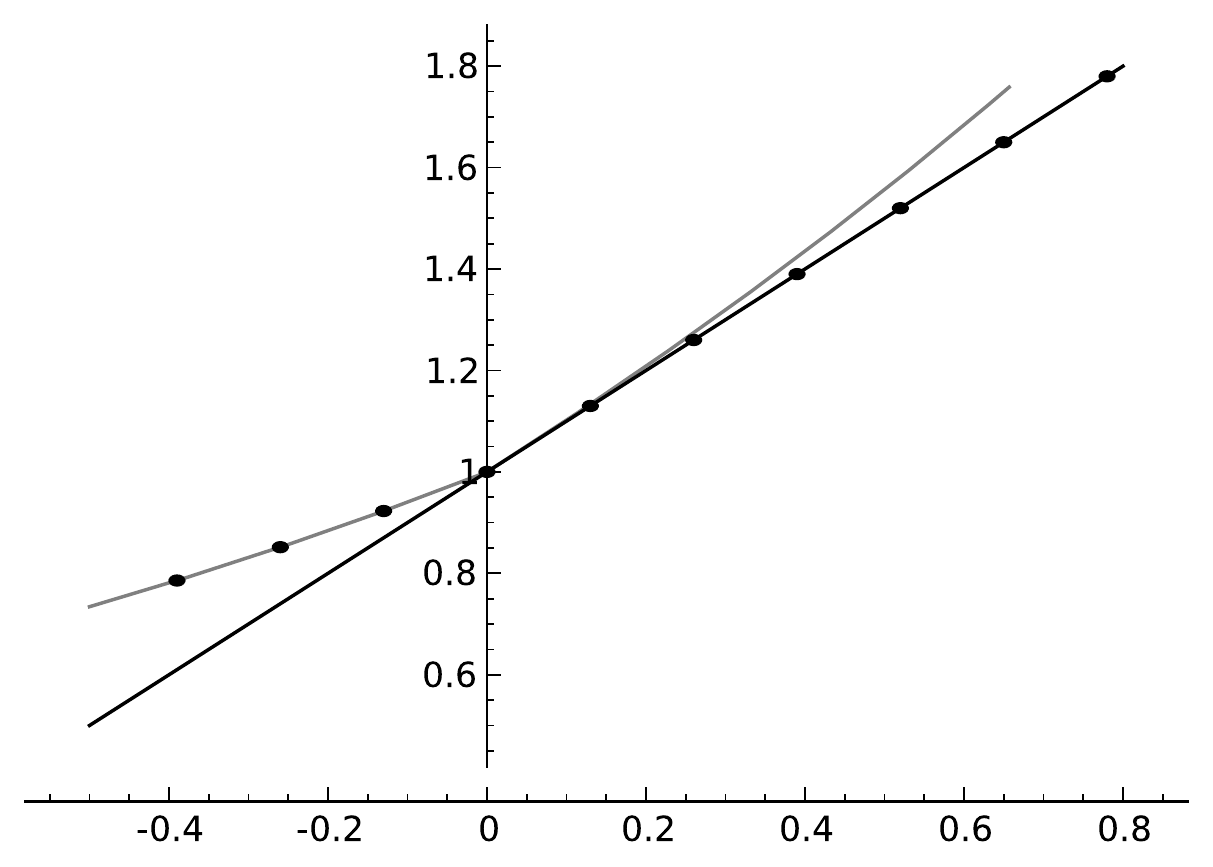}}}\\
\subfigure[\label{fig:medio}$\alpha=0.28$]{
\resizebox*{6.9cm}{!}{\includegraphics{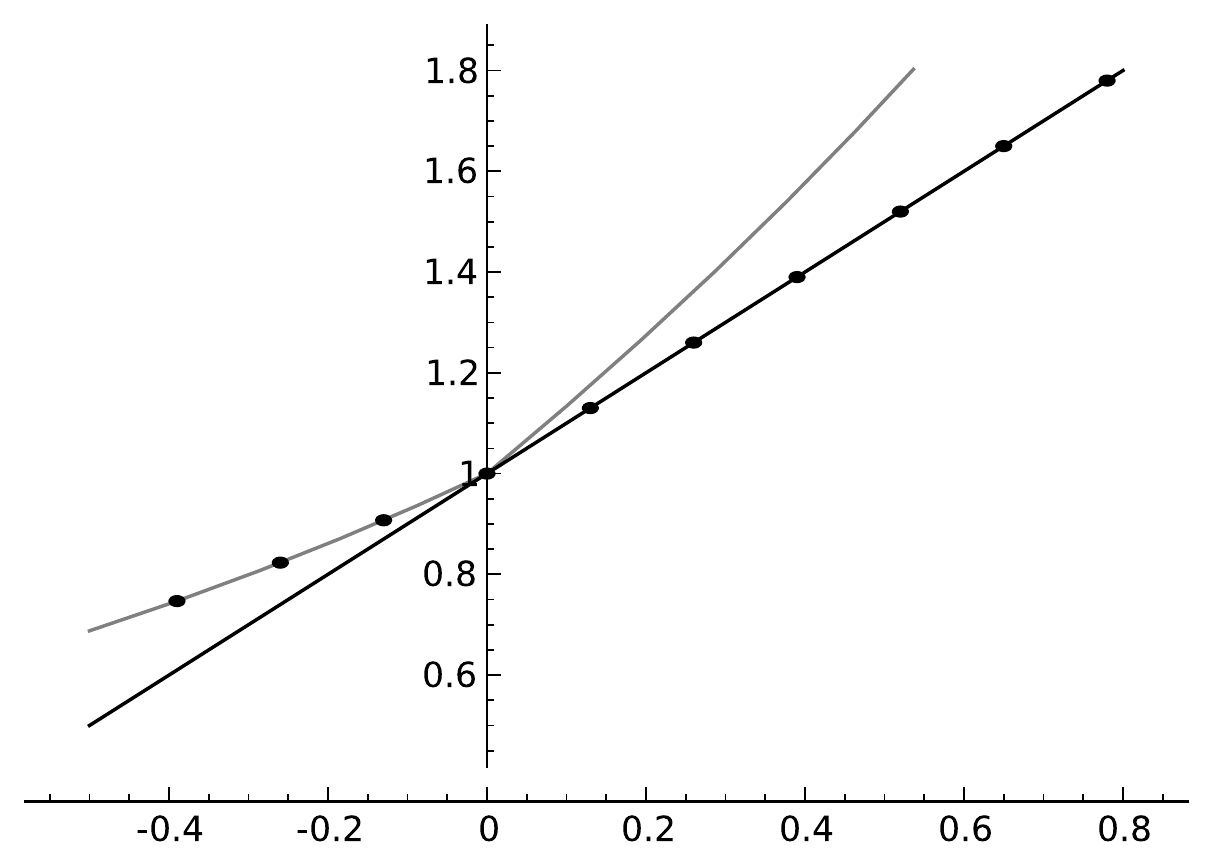}}}%
\subfigure[\label{fig:a2}$\alpha=\alpha_2=0.5$]{
\resizebox*{6.9cm}{!}{\includegraphics{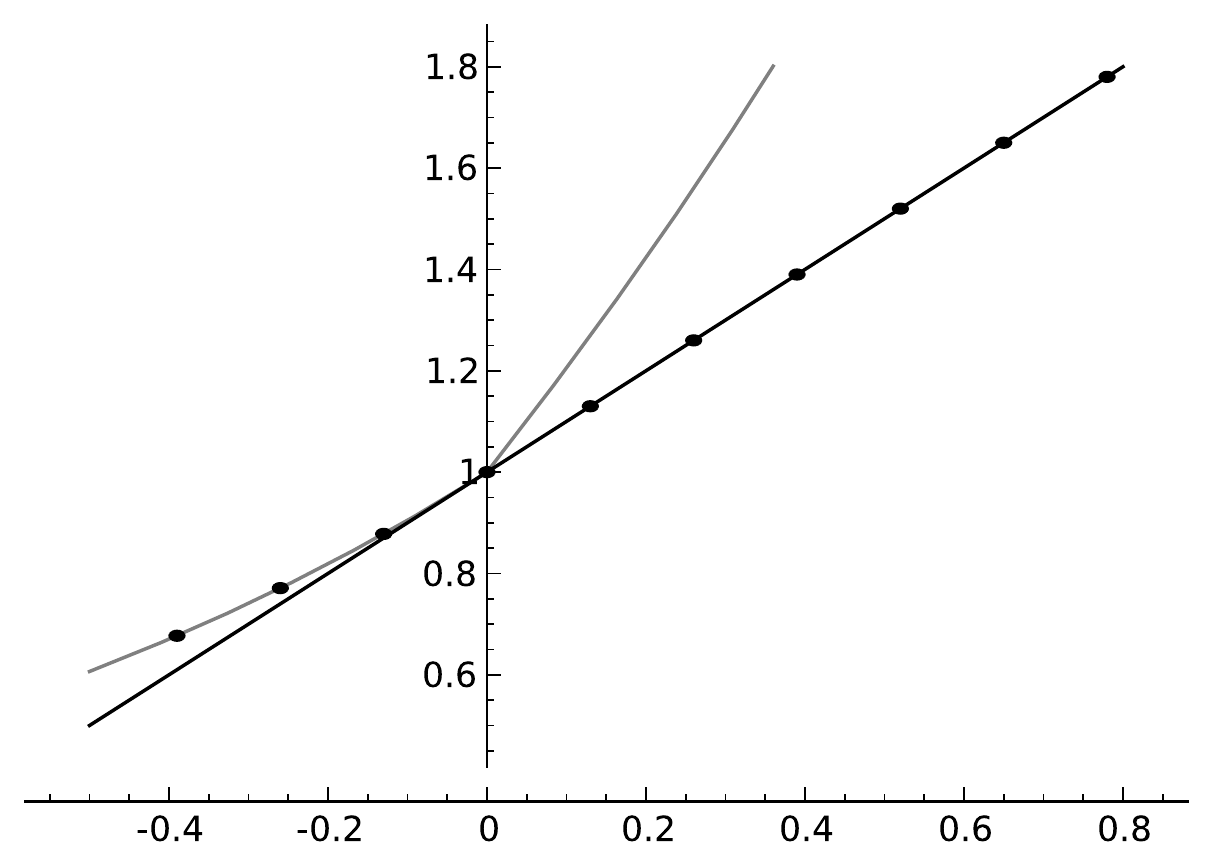}}}\\
\subfigure[\label{fig:mayor}$\alpha=2$]{
\resizebox*{6.9cm}{!}{\includegraphics{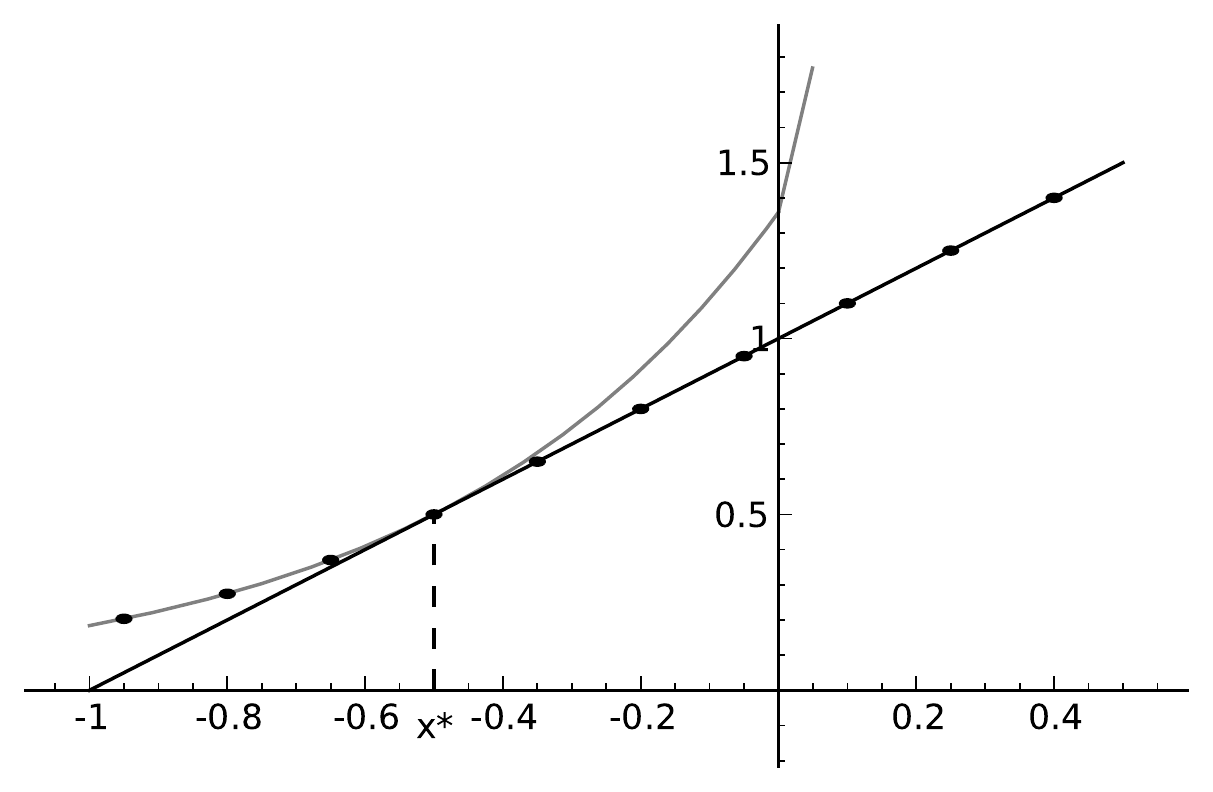}}}%
\end{center}
\caption{\label{fig:sticky}Solution of the OSP for the sticky Brownian motion with different values of $\alpha$ with reward function $(x+1)^+$ (graphic in black). The graphic in gray corresponds to $(cte) \psia$. The value function $\Va$ is indicated with dots, it coincides with $(cte) \psia$ for $x\leq x^*$ and with $g$ for $x\geq x^*$.}
\end{figure}


\begin{thebibliography}{10}
\providecommand{\url}[1]{\texttt{#1}}
\providecommand{\urlprefix}{URL }

\bibitem{bs}
F. Black and M. Scholes, \emph{The pricing of options and corporate
  liabilities}, The Journal of Political Economy  (1973), pp. 637--654.

\bibitem{borodin}
A.N. Borodin and P. Salminen, \emph{Handbook of {B}rownian motion---facts and
  formulae}, 2nd ed., Birkh{\"a}user Verlag, Basel,  2002.

\bibitem{dynkin:1963}
E.B. Dynkin, \emph{Optimal choice of the stopping moment of a {M}arkov
  process}, Dokl. Akad. Nauk SSSR 150 (1963), pp. 238--240.

\bibitem{dynkin:1969}
E.B. Dynkin, \emph{The exit space of a {M}arkov process}, Uspehi Mat. Nauk 24
  (1969), pp. 89--152.

\bibitem{dynkin:books}
E.B. Dynkin, \emph{Theory of {M}arkov processes}, Dover Publications Inc.,
  Mineola, NY,  2006, translated from the Russian by D. E. Brown and edited by
  T. K{\"o}v{\'a}ry, Reprint of the 1961 English translation.

\bibitem{elkaroui}
N. El~Karoui, J. Lepeltier, and A. Millet, \emph{A probabilistic approach of
  the reduite}, Probability and Mathematical Statistics 13 (1992), pp. 97--121.

\bibitem{feller1957generalized}
W. Feller, \emph{Generalized second order differential operators and their
  lateral conditions}, Illinois journal of mathematics 1 (1957), pp. 459--504.

\bibitem{grigelionis1966stefan}
B.I. Grigelionis and A.N. Shiryaev, \emph{On {S}tefan's problem and optimal
  stopping rules for {M}arkov processes}, Theory of Probability \& Its
  Applications 11 (1966), pp. 541--558.

\bibitem{itoMcKean}
K. It{\^o} and J.H.P. McKean, \emph{Diffusion processes and their sample
  paths}, Springer-Verlag, Berlin,  1974, second printing, corrected, Die
  Grundlehren der mathematischen Wissenschaften, Band 125.

\bibitem{jacka1993local}
S. Jacka, \emph{Local times, optimal stopping and semimartingales}, The Annals
  of Probability  (1993), pp. 329--339.

\bibitem{ks}
I. Karatzas and S.E. Shreve, \emph{Brownian motion and stochastic calculus},
  Graduate Texts in Mathematics, Vol. 113, 2nd ed., Springer-Verlag, New York,
  1991.

\bibitem{kw}
H. Kunita and T. Watanabe, \emph{Markov processes and {M}artin boundaries},
  Bull. Amer. Math. Soc. 69 (1963), pp. 386--391.

\bibitem{kw1}
H. Kunita and T. Watanabe, \emph{Markov processes and {M}artin boundaries.
  {I}}, Illinois J. Math. 9 (1965), pp. 485--526.

\bibitem{lejay}
A. Lejay, \emph{On the constructions of the skew {B}rownian motion}, Probab.
  Surv. 3 (2006), pp. 413--466,
  \urlprefix\url{http://dx.doi.org/10.1214/154957807000000013}.

\bibitem{mckean}
H. {McKean Jr}, \emph{Appendix: A free boundary problem for the heat equation
  arising from a problem in mathematical economics}, Industrial Management
  Review 6 (1965), pp. 32--39.

\bibitem{merton}
R.C. Merton, \emph{Theory of rational option pricing}, Bell J. Econom. and
  Management Sci. 4 (1973), pp. 141--183.

\bibitem{mordecki-salminen}
E. Mordecki and P. Salminen, \emph{Optimal stopping of {H}unt and {L}\'evy
  processes}, Stochastics 79 (2007), pp. 233--251,
  \urlprefix\url{http://dx.doi.org/10.1080/17442500601100232}.

\bibitem{oksendal}
B. {\O}ksendal, \emph{Stochastic differential equations}, sixth ed.,
  Springer-Verlag, Berlin,  2003, an introduction with applications.

\bibitem{peskir07}
G. Peskir, \emph{Principle of smooth fit and diffusions with angles},
  Stochastics 79 (2007), pp. 293--302,
  \urlprefix\url{http://dx.doi.org/10.1080/17442500601040461}.

\bibitem{peskir2012duality}
G. Peskir, \emph{A duality principle for the {L}egendre transform}, Journal of
  Convex Analysis 19 (2012), pp. 609--630.

\bibitem{ps}
G. Peskir and A. Shiryaev, \emph{Optimal stopping and free-boundary problems},
  Birkh\"auser Verlag, Basel,  2006.

\bibitem{revuz}
D. Revuz and M. Yor, \emph{Continuous martingales and {B}rownian motion},
  Grundlehren der Mathematischen Wissenschaften [Fundamental Principles of
  Mathematical Sciences], Vol. 293, 3rd ed., Springer-Verlag, Berlin,  1999.

\bibitem{salminen85}
P. Salminen, \emph{Optimal stopping of one-dimensional diffusions}, Math.
  Nachr. 124 (1985), pp. 85--101,
  \urlprefix\url{http://dx.doi.org/10.1002/mana.19851240107}.

\bibitem{salminen00RussianOptions}
P. Salminen, \emph{On {R}ussian options}, Theory of Stochastic Processes 6
  (2000), pp. 3--4.

\bibitem{samee2010principle}
F. Samee, \emph{On the principle of smooth fit for killed diffusions},
  Electronic Communications in Probability 15 (2010), pp. 89--98.

\bibitem{shepp93}
L.A. Shepp and A.N. Shiryaev, \emph{The {R}ussian option: reduced regret}, Ann.
  Appl. Probab. 3 (1993), pp. 631--640.

\bibitem{shepp94}
L.A. Shepp and A.N. Shiryaev, \emph{A new look at the ``{R}ussian option''},
  Teor. Veroyatnost. i Primenen. 39 (1994), pp. 130--149,
  \urlprefix\url{http://dx.doi.org/10.1137/1139004}.

\bibitem{shiryaev}
A.N. Shiryaev, \emph{Optimal stopping rules}, Stochastic Modelling and Applied
  Probability, Vol.~8, Springer-Verlag, Berlin,  2008, translated from the 1976
  Russian second edition by A. B. Aries, Reprint of the 1978 translation.

\bibitem{taylor}
H.M. Taylor, \emph{Optimal stopping in a {M}arkov process}, Ann. Math. Statist.
  39 (1968), pp. 1333--1344.

\bibitem{wald}
A. Wald, \emph{Sequential {A}nalysis}, John Wiley \& Sons Inc., New York,
  1947.

\end{thebibliography}
\end{document}